\documentclass[reqno,11pt]{amsart}
\newlength{\myhmargin}  
\usepackage[text={6in,8in},centering,letterpaper,dvips, marginparwidth=0.85in]{geometry}
\usepackage{mathtools,mathabx,amsfonts,amsmath,amssymb,amsthm}
\usepackage{listings,verbatim,setspace,latexsym,xspace}
\usepackage{upquote}
\usepackage[toc,title]{appendix}
\usepackage{color}
\usepackage{graphicx}
\usepackage{caption}
\usepackage{subcaption}
\usepackage[none]{hyphenat}
\usepackage{hyperref}
\usepackage{tikz-cd}
\usepackage{todonotes}
\usepackage[normalem]{ulem}

\DeclareMathAlphabet{\mathpzc}{OT1}{pzc}{m}{it}
\DeclareSymbolFont{bbold}{U}{bbold}{m}{n}
\DeclareSymbolFontAlphabet{\mathbbold}{bbold}

\lstdefinelanguage{Sage}[]{Python}
{morekeywords={True,False,sage,singular},
sensitive=true}
\definecolor{dblackcolor}{rgb}{0.0,0.0,0.0}
\definecolor{dbluecolor}{rgb}{.01,.02,0.7}
\definecolor{dredcolor}{rgb}{0.8,0,0}
\definecolor{dgraycolor}{rgb}{0.30,0.3,0.30}

\newcommand{\dblack}{\color{dblackcolor}\bf}

\renewcommand{\emph}[1]{{\dblack{#1}}}

\newcommand{\symgrp}{\mathrm{S}_N}

\newcommand{\spc}{\mathfrak{s}}
\newcommand{\F}{\mathbb{F}}
\newcommand{\cfhat}{\widehat{\mathit{CF}}}
\newcommand{\dhat}{\widehat{\partial}_{\mathit{HF}}}
\newcommand{\hfhat}{\widehat{\mathit{HF}}}
\newcommand{\del}{\partial} 
\newcommand{\CCF}{\widehat{\mathcal{CF}}}
\newcommand{\N}{\textsc{n}}
\newcommand{\pitwo}{\pi_2}
\newcommand{\pitwohat}{\widehat{\pi}_2}
\newcommand{\C}{\mathbf{x}_\xi} 

\newcommand{\chat}{\widehat{c}}

\newcommand{\R}{\mathbb{R}}

\newcommand{\x}{\mathbf{x}}
\newcommand{\y}{\mathbf{y}}

\newcommand{\bsa}{\boldsymbol\alpha}
\newcommand{\bsb}{\boldsymbol\beta}

\newcommand{\Mba}{\Sigma,\bsb,\bsa}

\newcommand{\arcs}{\mathbf{a}} 
\newcommand{\base}{\mathbf{z}}

\newcommand{\ordn}{o} 
\newcommand{\ord}{\mathbf{o}} 

\hypersetup{backref,colorlinks=true,citecolor=blue,linkcolor=blue}
\makeatletter
\def\@xfootnote[#1]{%
  \protected@xdef\@thefnmark{#1}%
  \@footnotemark\@footnotetext}
\makeatother

\title[Computing Heegaard Floer invariants]{Computing Heegaard Floer invariants of closed contact 3-manifolds from open books}
\author[Kutluhan]{\c Ca\u gatay Kutluhan}
\author[Mati\'{c}]{Gordana Mati\'{c}}
\author[Van Horn-Morris]{Jeremy Van Horn-Morris}
\author[Wand]{Andy Wand}
\address{Department of Mathematics, University at Buffalo}
\email{kutluhan@buffalo.edu}
\address{Department of Mathematics, University of Georgia}
\email{gordana@math.uga.edu}
\address{Department of Mathematical Sciences, University of Arkansas}
\email{jvhm@uark.edu}
\address{School of Mathematics and Statistics, University of Glasgow}
\email{andy.wand@glasgow.ac.uk}

\newtheorem{theorem}{Theorem}
\newtheorem*{thm}{Theorem}

\newtheorem{lem}[theorem]{Lemma}

\theoremstyle{definition}

\numberwithin{equation}{section}

\begin{document}
\sloppy
\bibliographystyle{amsalpha}

\begin{abstract}
We present two SageMath programs that build on and improve upon Sucharit Sarkar's \verb|hf-hat|. Given an abstract open book and a collection of pairwise disjoint properly embedded arcs on a page of the open book, the first program, \verb|hf-hat-obd|, can be used to analyze the resulting Heegaard diagram, while the second, \verb|hf-hat-obd-nice| computes the hat version of Heegaard Floer homology of the closed oriented 3-manifold described by the Heegaard diagram as long as the latter is nice. We also provide an auxiliary program, \verb|makenice|, that can be used to produce a nice Heegaard diagram out of any abstract open book and a collection of pairwise disjoint properly embedded arcs on a page of the open book. The primary applications of \verb|hf-hat-obd-nice| are to the computation of the Ozsv\'ath--Szab\'o contact invariant and to the detection of finiteness of spectral order, which is a Stein fillability obstruction that is stronger than the vanishing of the Ozsv\'ath--Szab\'o contact invariant.
\end{abstract}

\maketitle

\section{Introduction}
A classical result due to Alexander \cite{Alexander} states that every closed oriented 3-manifold admits an open book decomposition. Once an open book and an arc basis on its page are chosen, we can construct a Heegaard diagram. The latter can then be used to compute the Heegaard Floer homology of the 3-manifold described by that Heegaard diagram. Moreover, the chosen open book supports a contact structure \cite{TW} and an invariant of this contact structure living in Heegaard Floer homology defined by Ozsv\'ath and Szab\'o in \cite{OzsvathSzabo4} is represented by a distinguished tuple of points in that Heegaard diagram \cite{HKM}. Using this description of the contact invariant in Heegaard Floer homology, the authors defined in \cite{KMVHMW1} a refinement of it, called \textit{spectral order}, taking values in the set $\mathbb{N}\cup\{\infty\}$, and showed that 
\begin{thm}{\cite[Corollary 4.8]{KMVHMW1}}
If the contact structure is Stein fillable, then its spectral order is infinite.
\end{thm}
\noindent Moreover, the spectral order of the contact structure can be computed using a sufficiently large collection of pairwise disjoint properly embedded arcs on a page of the open book and it is finite if and only if it is finite for some arc collection. In order to compute spectral order for a fixed arc collection, one has to
\begin{enumerate}\leftskip-0.25in
\item \label{item-1}find all Maslov index-1 positive domains in the Heegaard diagram, and
\item \label{item-2}compute the number of holomorphic curves in the moduli spaces of all Maslov index-1 positive domains.
\end{enumerate}
\newpage  
Task (\ref{item-1}) is in general impossible to fulfill if the Heegaard diagram belongs to a closed oriented 3-manifold with non-zero first Betti number. This is why Sucharit Sarkar's \verb|hf-hat| \cite{Sarkar-github} can only be used to analyze Heegaard diagrams for rational homology 3-spheres. In Section \ref{ssec:bound}, we show that there is an \textit{a priori} bound on the number of positive domains between any pair of generators of the Heegaard Floer chain complex when the Heegaard diagram is obtained from an open book and a collection of pairwise disjoint properly embedded arcs on a page of the open book. This allows us to find all positive domains between pairs of generators of the Heegaard Floer chain complex algorithmically, thereby completing task (\ref{item-1}). Using this fact, we wrote an improved version of \verb|hf-hat|, named \verb|hf-hat-obd|, a tutorial for which is provided in Section \ref{sec:hf-hat-obd}. As for task (\ref{item-2}), it was shown by Plamenevskaya in \cite{Plamenevskaya} that one may isotope the monodromy of the open book so as to produce a nice Heegaard diagram. Her proof is based on an algorithm of Sarkar and Wang \cite[Section 4]{SarkarWang} which we implement in an auxiliary program, \verb|makenice|. Once a nice Heegaard diagram resulting from an open book and a collection of pairwise disjoint properly embedded arcs on a page of the open book is provided, \verb|hf-hat-obd-nice| can be used to compute the hat version of Heegaard Floer homology of the 3-manifold described by the Heegaard diagram and the Ozsv\'ath--Szab\'o contact invariant of the contact structure supported by the open book. In addition, \verb|hf-hat-obd-nice| implements an algorithm to compute the value of spectral order of the open book and the arc collection based on a criterion for finiteness of spectral order, which we prove in Section~\ref{ssec:finiteorder}. A tutorial for \verb|hf-hat-obd-nice| is provided in Section \ref{sec:hf-hat-obd-nice}. 
\subsection*{Acknowledgements}
\c Ca\u gatay Kutluhan was supported by Simons Foundation grant No.~519352. Gordana Mati\'c was supported by NSF grants DMS-1664567 and DMS-1612036. Jeremy Van Horn-Morris was supported by Simons Foundation grant No.~279342. Andy Wand was supported by EPSRC grant EP/P004598/1.

\section{Main Results}
\label{sec:keylemmas}
In this section, we prove some key results that form the basis of the algorithms used in \verb|hf-hat-obd| and \verb|hf-hat-obd-nice|. To set the stage, let $(S,\varphi)$ be an abstract open book and $\arcs=\{a_1,\dots,a_\N\}$  be a collection of pairwise disjoint, properly embedded oriented arcs on $S$ cutting it open into a disjoint union of disks. Then the triple $(S,\varphi,\arcs)$, which we shall refer to as an \textit{open book diagram}, yields a weakly admissible (multi-)pointed Heegaard diagram $(\Sigma,\bsb,\bsa,\base)$ where
\begin{itemize}\leftskip-0.25in
\item $\Sigma=S_{1/2}\cup_{\partial S}-S_0$,
\item $\bsa=\{\alpha_1,\dots,\alpha_\N\}$ and $\alpha_i = a_i\cup-a_i$,
\item $\bsb=\{\beta_1,\dots,\beta_\N\}$ and $\beta_i = b_i\cup-\varphi(b_i)$ with each $b_i$ a Hamiltonian isotopic copy of $a_i$ that intersects $a_i$ positively in one distinguished point ${x_\xi}^i\in\mathrm{int}(S_{1/2})$, called a \textit{contact intersection point}, and
\item $\base=\{z_1,\dots, z_m\}$ and each $z_i$ lies in the interior of a distinct connected component of $S_{1/2}\, \setminus\bigcup_{i=1}^{\N} a_i$, outside the thin strips between $a_i$ and $b_i$. 
\end{itemize}

\newpage
\noindent The $\bsa$- and $\bsb$-curves divide the Heegaard surface $\Sigma$ into a number of subsurfaces whose closures are called \textit{regions}. Let $\x=\{x_1,\dots,x_\N\},\y=\{y_1,\dots,y_\N\}$ be two sets of intersection points between $\bsa$- and $\bsb$-curves each of which determines a one-to-one pairing between the two kinds of curves, and $\pitwohat(\x,\y)$ denote the set of 2-chains $D$ that are formal integer linear sums of unpointed regions in the Heegaard diagram $(\Sigma,\bsb,\bsa,\base)$ such that $\partial(\partial|_{\bsb}D)=\y-\x$, where $\x$ and $\y$ are considered 0-chains. Such 2-chains are called \textit{domains} and a domain $D$ is \textit{positive} ($D\geq0$) if its multiplicity in every unpointed region is non-negative. Elements of $\pitwohat(\x,\x)$ are called \textit{periodic domains}.

\subsection{A bound on the number of positive domains}
\label{ssec:bound}
First we prove the aforementioned \textit{a priori} bound on the number of positive domains between $\x$ and $\y$. The upcoming lemma can be considered an improvement over \cite[Lemma 4.13]{OzsvathSzabo} for Heegaard diagrams resulting from open books. 
\begin{lem}
\label{lem:lem-a}
Suppose that $\pitwohat(\x,\y)\neq\emptyset$ and let $k$ denote the number of contact intersection points in $\y$ that are not in $\x$. Then there are at most $2^k$ positive domains in $\pitwohat(\x,\y)$.
\end{lem}
\begin{proof}
Fix $D_\circ\in\pitwohat(\x,\y)$ and let $D\in \pitwohat(\x,\y)$ be a positive domain. If $D\neq D_\circ$, then $D-D_\circ\in\pitwohat(\C,\C)$ is a non-trivial periodic domain. Since a non-trivial periodic domain has to have the entirety of at least one $\bsa$- or $\bsb$-curve with non-zero multiplicity on its boundary, it has to contain at least one contact intersection point on its boundary. Moreover, a contact intersection point ${x_\xi}^i$ is at the corner of exactly two unpointed regions, ${R^i}_1$ and ${R^i}_3$, which intersect only at that contact intersection point (see Figure \ref{fig:regions}); therefore, a non-trivial periodic domain containing the contact intersection point ${x_\xi}^i$ has to have both $\alpha_i$ and $\beta_i$ on its boundary with multiplicities in ${R^i}_1$ and ${R^i}_3$ adding up to zero. It follows, by applying \cite[Lemma 2.18]{OzsvathSzabo}, that the multiplicities of $D-D_\circ$ in ${R^i}_1$ and ${R^i}_3$ for $i\in\{1,\dots,\N\}$ uniquely determine the periodic domain. Consequently, multiplicities of $D$ in ${R^i}_1$ and ${R^i}_3$ for ${x_\xi}^i$ in $\y$ but not in $\x$ uniquely determine that domain. Since $D$ has multiplicity $1$ in exactly one of ${R^i}_1$ or ${R^i}_3$ while having multiplicity $0$ in the other three regions with corners at ${x_\xi}^i$ for each ${x_\xi}^i$ in $\y$ but not in $\x$, there are at most $2^k$ positive domains in $\pitwohat(\x,\y)$.
\end{proof}

\begin{figure}[h]
\centering
\includegraphics[width=1.5in]{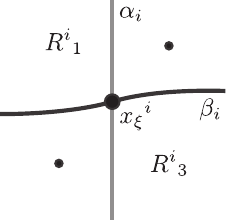}
\caption{Neighborhood of a contact intersection point.}
\label{fig:regions}
\end{figure}

\subsection{Detecting finiteness of spectral order}
\label{ssec:finiteorder}
As was shown in \cite[Section 5.2]{KMVHMW1}, finiteness of spectral order is an obstruction to Stein fillability of contact 3-manifolds that is stronger than vanishing of the Ozsv\'ath--Szab\'o contact invariant. In this regard, the ability to detect finiteness of spectral order is valuable even if we may not be able to compute its exact value. To detect finiteness of spectral order, it suffices to find an open book $(S,\phi)$ and an arc collection $\arcs$ on $S$ for which spectral order is finite. In this subsection, we discuss how to do this efficiently. 

In what follows, $Y$ denotes the closed oriented 3-manifold described by the Heegaard diagram $(\Sigma,\bsa,\bsb)$ and $\xi$ is the contact structure on $Y$ supported by the open book $(S,\varphi)$. We start by recalling the definition of spectral order. For any $D\in \pitwohat(\x,\y)$, define
\begin{equation}
\label{eq:jplus-prod}
J_+(D):=\mu(D)-2e(D)+|\x|-|\y|,
\end{equation}
where 
\begin{itemize}\leftskip-0.25in
\item $|\cdot|$ denotes the number of disjoint cycles in the element of the symmetric group $\symgrp$ determined by the pairing between the $\bsa$- and $\bsb$-curves associated to a generator of the Heegaard Floer chain complex (e.g. the canonical generator $\x_\xi=\{{x_\xi}^1,\dots,{x_\xi}^\N\}$ has $|\x_\xi|=\N$), 
\item $e(D)$ is the Euler measure of $D$ (see \cite[\textsection 4.1 p. 973]{Lipshitz} for definition), and 
\item $\mu(D)$ is the Maslov index of $D$, which, by a formula of Lipshitz \cite[Corollary 4.10]{Lipshitz} (cf.  \cite[Proposition $4.8'$]{Lipshitz2}), is equal to 
\begin{equation}
\label{eq:maslov}
\mu(D)=n_{\x}(D)+n_{\y}(D)+e(D),
\end{equation}
with $n_\x(D)$ representing the point measure of $D$ at $\x$, i.e. the sum of the averages of the coefficients of $D$ in the four regions with corners at each $x_i\in\alpha_i\cap\beta_j$.
\end{itemize} 
It was shown in \cite[Section 2.2]{KMVHMW1} that $J_+(D)$ is a non-negative integer divisible by two when $D\in\pitwohat(\x,\y)$ is represented by a holomorphic curve. 
Therefore, we can decompose the Heegaard Floer differential as
\[\dhat=\partial_0+\partial_1+\cdots+\partial_\ell+\cdots,\]
where $\partial_\ell$ counts holomorphic curves representing domains with $J_+=2\ell$. Using this decomposition, construct a chain complex
\[\CCF(S,\phi,\arcs):=\cfhat(\Mba)\otimes_\F \F[t,t^{-1}]\] 
over $\F=\mathbb{Z}/2\mathbb{Z}$ filtered by the power of $t$ and endowed with the differential 
\[\widehat{\partial}:=\sum_{\ell\in\mathbb{N}}\partial_\ell \otimes t^{-\ell}.\]
Define $\ordn(S,\phi,\arcs)$ to be the smallest non-negative integer $k$ such that $\x_\xi$ represents the trivial class in the $k+1^{\textrm{st}}$ page of the spectral sequence $E^{\ast}(S,\phi,\arcs)$ associated to the filtered chain complex $(\CCF(S,\phi,\arcs),\widehat{\partial})$, and the spectral order $\ord(\xi)$ to be
\[\ord(\xi):=\min\{\ordn(S,\phi,\arcs)\},\]
where the minimum is taken over all open books supporting $\xi$ and all collections of pairwise disjoint properly embedded arcs on pages of those open books. With the preceding understood, we have the following lemma.
\begin{lem}
\label{lem:finiteness}
Suppose that $c_1(\xi)$ is torsion. If $\chat(\xi)=0$, then $\ord(\xi)<\infty$.
\end{lem}
\begin{proof}
First note that for any periodic domain $P\in\pitwohat(\C,\C)$, we have $\mu(P)=\langle c_1(\xi), H(P)\rangle=0$, and $n_{\C}(P)=0$; hence, $e(P)=0$. As a result, $J_+(P)=0$ for every periodic domain $P\in\pitwohat(\C,\C)$. It follows that there is a well-defined filtration $\mathcal{F}$ on $\cfhat(\Sigma,\bsb,\bsa,\spc_\xi)$ by $J_+$ where $\C$ lies in $\mathcal{F}_0\setminus\mathcal{F}_{-1}$. Moreover, this filtration is bounded by virtue of the fact that $\cfhat(\Sigma,\bsb,\bsa,\spc_\xi)$ is finitely generated. Consequently, the spectral sequence of the filtered chain complex converges to the homology of $\cfhat(\Sigma,\bsb,\bsa,\spc_\xi)$, that is, $\hfhat(-Y,\spc_\xi)$. Meanwhile, $\ord(\xi)<\infty$ if and only if $\C$ represents the zero class on some finite page of the spectral sequence of the filtered chain complex, which must be the case if $\chat(\xi)=0$ in $\hfhat(-Y,\spc_\xi)$.
\end{proof}

To detect whether $\ordn(S,\phi,\arcs)<\infty$ or not in general, we have to compute the differential on $\cfhat(\Sigma,\bsb,\bsa,\spc_\xi)$. To be able to do this, we assume that $\phi$ is such that the Heegaard diagram $(\Sigma,\bsb,\bsa,\base)$ constructed from $(S,\phi,\arcs)$ is nice, i.e. all unpointed regions in the Heegaard diagram are either bigons or rectangles. Since Maslov index-1 positive domains in a nice Heegaard diagram are either empty embedded bigons or empty embedded rectangles \cite[Theorem 3.3]{SarkarWang}, the moduli space of holomorphic curves for a Maslov index-1 positive domain consists of a single point, and 
\[\dhat=\partial_0+\partial_1.\]
Moreover, $\del_0$ and $\del_1$ are both differentials themselves and $\del_0\del_1=\del_1\del_0$. With the preceding understood, let $C_0\subset \cfhat(\Sigma,\bsb,\bsa,\spc_\xi)$ denote the subspace of chains in the same (relative) Maslov grading as $\x_\xi$ and $C_1\subset \cfhat(\Sigma,\bsb,\bsa,\spc_\xi)$ denote the subspace of chains in (relative) Maslov grading 1 above $\x_\xi$. Then $\ordn(S,\phi,\arcs)<\infty$ if and only if there exist chains $b_0,\dots,b_k\in C_1$ with $k\geq 0$ such that
\begin{equation}
\label{eq:zigzag}
\begin{split}
\del_0b_0+\del_1b_1&=\x_\xi,\\
\del_0b_1+\del_1b_2&=0,\\
&\vdots\\
\del_0b_{k-1}+\del_1b_k&=0,\\
\del_0b_k&=0.
\end{split}
\end{equation}
\begin{figure}[h]
\centering
\includegraphics[width=4in]{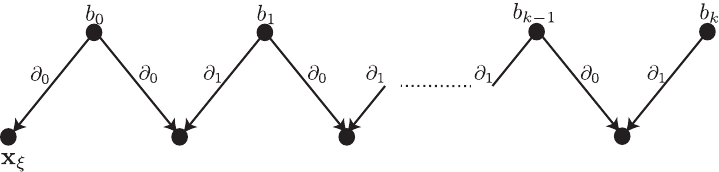}
\caption{Directed graph depicting $b_0,\dots,b_k\in C_1$ from \eqref{eq:zigzag} and their differentials.}
\label{fig:zigzag}
\end{figure}
\vspace{-0.1in}

We shall regard $\del_0,\del_1$ as maps from $C_1$ to $C_0$. The next lemma allows us to assume, without loss of generality, that $\del_1$ is injective.
\begin{lem}
\label{lem:injdel1}
There exists a minimal subspace $K\subset C_1$ with the property that the map from $C_1/K$ to $C_0/\del_0(K)$ induced by $\del_1$ is injective.
\end{lem}
\begin{proof}
Construct a nested sequence $K_0\subset K_1\subset\cdots\subset K_i\subset\cdots$ of subspaces of $C_1$ as follows:
\begin{eqnarray*}
K_0&=&\ker(\del_1:C_1\to C_0),\\
K_i&=&\ker(\del_1:C_1\to C_0/\del_0(K_{i-1}))\;\textrm{for\;each}\; i\geq 1.
\end{eqnarray*}
Since $C_1$ is finite dimensional, there exists $n\geq 0$ and a subspace $K\subset C_1$ such that $K_i=K$ for all $i\geq n$. Then 
\[\ker(\del_1:C_1\to C_0/\del_0(K_n))=K,\]
hence the induced map $\underline{\del}\,_1:C_1/K\to C_0/\del_0(K))$ is injective.
\end{proof}
\noindent Note that if $\ordn(S,\phi,\arcs)$ is non-zero, $K=C_1$ implies that $\ordn(S,\phi,\arcs)=\infty$ since otherwise $\del_1(C_1)\subset \del_0(C_1)$ and we would be able to find $b\in C_1$ such that $\del_0 b=\x_\xi$, resulting in a contradiction.

With the above understood, if $\ordn(S,\phi,\arcs)\neq 0$ and $K\neq C_1$, then $C_0/\del_0(K)$ and $C_1/K$ are both non-zero and we may instead consider the maps 
\begin{eqnarray*}
&\underline{\del}\,_0:C_1/K\to C_0/\del_0(K)),\\
&\underline{\del}\,_1:C_1/K\to C_0/\del_0(K)).
\end{eqnarray*}
Clearly, if there exist $b_0,\dots, b_k\in C_1$ satisfying \eqref{eq:zigzag}, then $\underline{b}\,_0,\dots,\underline{b}\,_k\in C_1/K$ satisfy 
\begin{equation}
\label{eq:zigzagmod}
\begin{split}
\underline{\del}\,_0\,\underline{b}\,_0+\underline{\del}\,_1\,\underline{b}\,_1&=\underline{\x}\,_\xi,\\
\underline{\del}\,_0\,\underline{b}\,_1+\underline{\del}\,_1\,\underline{b}\,_2&=0,\\
&\vdots\\
\underline{\del}\,_0\,\underline{b}\,_{k-1}+\underline{\del}\,_1\,\underline{b}\,_k&=0,\\
\underline{\del}\,_0\,\underline{b}\,_k&=0.
\end{split}
\end{equation}
Conversely,
\begin{lem}
\label{lem:reduction}
If there are $\underline{b}\,_0,\dots,\underline{b}\,_k\in C_1/K$ satisfying \eqref{eq:zigzagmod}, then there are $b_0,\dots,b_k\in C_1$ satisfying \eqref{eq:zigzag}.
\end{lem}
\begin{proof}
Let $\underline{b}\,_0,\dots,\underline{b}\,_k\in C_1/K$ be such that \eqref{eq:zigzagmod} is satisfied. Then there are $c_0,\dots,c_k\in C_1$ with $\underline{b}\,_i=c_iK$ for all $i=0,\dots,k$ and $d_0,\dots, d_k\in K$ such that
\begin{equation*}
\begin{split}
\del_0c_0+\del_1c_1&=\x_\xi+\del_0d_0,\\
\del_0c_1+\del_1c_2&=\del_0d_1,\\
&\vdots\\
\del_0c_{k-1}+\del_1c_k&=\del_0d_{k-1},\\
\del_0c_k&=\del_0d_k.
\end{split}
\end{equation*}
(We use the fact that $\del_0\del_1=\del_1\del_0$ to derive the above set of equations.) Let $b_k=c_k+d_k$. Then $\del_0 b_k=0$ and $\del_1 b_k=\del_1c_k+\del_1 d_k=\del_1c_k+\del_0e_{k-1}$ for some $e_{k-1}\in K$ by the proof of Lemma \ref{lem:injdel1}. Next, define $b_{k-1}=c_{k-1}+d_{k-1}+e_{k-1}$ and note that
\[\del_0b_{k-1}+\del_1b_k=\del_0(c_{k-1}+d_{k-1}+e_{k-1})+\del_1c_k+\del_0e_{k-1}=0.\]
Continue in the same way so as to define all other $b_i\in C_1$ and note that $b_0,\dots,b_k$ satisfy \eqref{eq:zigzag}.
\end{proof}
\noindent For the remainder of this section, we shall assume that $\del_1:C_1\to C_0$ is injective. Furthermore, if $\del_0:C_1\to C_0$ were also injective, then $\ordn(S,\phi,\arcs)$ would be infinity, as is easily seen from \eqref{eq:zigzag}. Therefore, we shall also assume that $\del_0$ is not injective.

Since $\del_1$ is injective, there exists a map $\delta: U_0=\del_1(C_1)\to C_0$ defined by $\delta = \del_0\circ{\del_1}^{-1}$. Note that $\ker{\delta}=\del_1(\ker{\del_0})\neq 0$ since $\del_0$ is not injective and $\del_1$ is injective. Next, for $i\geq 1$ define subspaces $U_i\subset C_0$ iteratively as $U_i=\delta^{-1}(U_{i-1})$. Note that these subspaces satisfy
\[\ker{\delta}\subset\cdots\subset U_i\subset\cdots\subset U_1\subset U_0,\]
and that $\bigcap U_i=U_n\neq 0$ for some $n\geq 0$ since $C_0$ is finite dimensional and $\ker\delta\neq 0$. Then for each $i\geq 1$ the maps defined by $\delta_i = \delta|_{U_i}: U_i\to U_{i-1}$ and 
\[\delta^i=\delta\circ\delta_1\circ\cdots\circ\delta_{i-1}:U_{i-1}\to C_0\]
have $\ker{\delta_i}=\ker\delta$ and 
\[\ker{\delta^i}={\delta_{i-1}}^{-1}({\delta_{i-2}}^{-1}(\dots({\delta_1}^{-1}(\ker\delta)))\cdots)={\delta_{i-1}}^{-1}(\ker{\delta^{i-1}}).\]
It is easy to see by induction that $\ker{\delta^i}\subset\ker{\delta^{i+1}}$ for all $i\geq 1$. Therefore, 
\begin{itemize}\leftskip-0.25in
\item $\ker{\delta^i}\subset U_n$ for all $i\geq 1$, and 
\item there exists $m\geq1$ such that $\ker{\delta^i}\subset\ker{\delta^m}$ for all $i\geq 1$.
\end{itemize}
Finally, let 
\[\mathcal{R}=\x_\xi+\del_0(C_1),\]
which we call the \textit{remainder set}. Then, 
\begin{theorem}
\label{thm:detectfiniteorder}
$\ordn(S,\phi,\arcs)<\infty$ if and only if $\ker{\delta^m}\cap\mathcal{R}\neq\emptyset$, where $m\geq1$ is as in the second bullet point above.
\end{theorem}
\begin{proof}
It is clear that if there exist $b_0,\dots,b_k\in C_1$ satisfying \eqref{eq:zigzag}, then $\x_\xi+\del_0 b_0\in \ker{\delta^k}\cap\mathcal{R}$ which in turn is contained in $\ker{\delta^m}\cap\mathcal{R}$ by the second bullet point above. Conversely, if there exist $r\in \ker{\delta^m}\cap\mathcal{R}$, then $\delta^k r=0$ for some $k\geq1$, $\x_\xi+r=\del_0b_0$ for some $b_0\in C_1$, and $b_i\in C_0$ defined by $b_i={\del_1}^{-1}r_i$ where $r_1=r$ and $r_{i+1} =\delta r_i$ for $i=1,\dots,k-1$ satisfy \eqref{eq:zigzag}. Hence, $\ordn(S,\phi,\arcs)<\infty$ and 
\[\ordn(S,\phi,\arcs)=\min\{k\;|\;\ker{\delta^k}\cap\mathcal{R}\neq\emptyset\}.\qedhere\]
\end{proof}


\section{A Tutorial for \texttt{hf-hat-obd}}
\label{sec:hf-hat-obd}
Download \verb|hf-hat-obd.sage| from \cite{KMVHMW-github}. After running SageMath on your computer, load the program:
\begin{lstlisting}
sage: load('hf-hat-obd.sage')
\end{lstlisting}
\noindent The primary input for the program is a list of regions cut out by the alpha and beta curves in the (multi-)pointed Heegaard diagram resulting from an open book and a collection of disjoint properly embedded arcs on a page of the open book, which we shall refer to as an \textit{open book diagram}. To form a list of regions, one starts by labeling the intersection points on each alpha curve making sure that the lowest index on each alpha curve is assigned to the contact intersection point on that curve. Once every intersection point is labeled, proceed as described by Sarkar in his SageMath program \verb|hf-hat| to associate to each region a list of intersection points where the first two intersection points in the list lie on the same alpha curve and the intersection points are ordered so as to traverse the boundary of that region in the counter-clockwise orientation. Finally, form a list of such lists of intersection points where the lists of intersection points corresponding to the pointed regions are located at the end. For example\footnote{This is the region list for the `warm-up' example presented in \cite[Section 5.1]{KMVHMW1}.};

\begin{lstlisting}
sage: rlist = [
    [1,0,19,20,13,14],
    [2,1,14,15],
    [3,2,15,16],
    [4,3,16,17],
    [5,4,17,6],
    [7,6,1,2],
    [8,7,2,3],
    [9,8,3,4],
    [10,9,4,5],
    [11,10,5,0,25,24],
    [12,11,24,23],
    [13,12,23,22,11,12,21,22],
    [14,13,22,23],
    [15,14,23,24],
    [16,15,24,25],
    [17,16,25,18],
    [19,18,7,8],
    [20,19,8,9],
    [21,20,9,10],
    [22,21,10,11],
    [20,21,12,13],
    [0,1,6,17,18,19,0,5,6,7,18,25]]
\end{lstlisting}
\newpage
Next, initialize the Heegaard Diagram with the given regions list. 
\begin{lstlisting}
sage: G1 = HeegaardDiagram(rlist, 1, 'NotStein')
\end{lstlisting}
\noindent Here the second entry indicates the number of base-pointed regions while the last entry provides a name to be appended to any output file that would result from running functions belonging to the HeegaardDiagram class.

Once the Heegaard diagram has been initialized, the following variables from \verb|hf-hat| are readily available.
\begin{enumerate}
\item \verb|G1.boundary_intersections|
\item \verb|G1.regions|
\item \verb|G1.regions_un|
\item \verb|G1.intersections|    
\item \verb|G1.euler_measures_2|
\item \verb|G1.euler_measures_2_un|
\item \verb|G1.boundary_mat|
\item \verb|G1.boundary_mat_un|
\item \verb|G1.point_measures_4|
\item \verb|G1.point_measures_4_un|
\item \verb|G1.alphas|
\item \verb|G1.betas|
\item \verb|G1.intersections_on_alphas|
\item \verb|G1.intersections_on_betas|
\item \verb|G1.intersection_incidence|
\item \verb|G1.intersection_matrix|
\item \verb|G1.generators|
\item \verb|G1.generator_reps|
\end{enumerate}

\noindent Also readily available are the following variables new to \verb|hf-hat-obd|.
\begin{enumerate}
\setcounter{enumi}{19}
\item \verb|G1.contact_intersections| stores the list of distinguished intersection points that give the Honda-Kazez-Matic representative for the Ozsv\'ath-Szab\'o contact class.
\item \verb|G1.periodic_domains| stores the free abelian group of periodic domains.
\item \verb|G1.pd_basis| stores a basis for the free abelian group of periodic domains.
\item \verb|G1.b1| stores the first Betti number of the Heegaard Diagram. (Note that this is strictly bigger than the first Betti number of the underlying 3-manifold if the diagram has more than one basepoint.)
\item \verb|G1.QSpinC| stores a list whose i-th entry is the list of Chern numbers for the SpinC structure associated to the i-th generator
\end{enumerate}
The following variables will be initialized once one runs \verb|G1.sort_generators()|:
\begin{enumerate}
\setcounter{enumi}{25}
\item \verb|G1.SpinC_structures| stores SpinC structures realized by Heegaard Floer generators, as specified by their Chern numbers and distinguished integers.
\item \verb|G1.SpinC| stores a list whose i-th entry is the SpinC structure of the i-th generator.
\item \verb|G1.genSpinC| stores a list whose i-th entry is a list of generators living in the i-th SpinC structure.
\item \verb|G1.gradings| stores a list whose i-th entry is a dictionary assigning to a generator in the i-th SpinC structure its (relative) Maslov grading (modulo divisibility of the i-th SpinC structure as stored in \verb|G1.div|.)
\item \verb|G1.differentials| is a dictionary of Maslov index 1 positive domains between pairs of generators.
\end{enumerate}

\noindent To complete the initialization of the Heegaard Floer chain complex, run 
\begin{lstlisting}
sage: G1.print_differentials()
\end{lstlisting}
which also outputs a \verb|.txt| file (\verb|NotStein_possible_differentials.txt|) with all possible Heegaard Floer differentials in every SpinC structure and fills in the dictionary \verb|G1.pos_differentials| that associates to a pair of Heegaard Floer generators the list of positive Maslov index 1 domains, if any, between them.

Below are a list of functions that can be used to analyze the Heegaard Floer chain complex once the Heegaard diagram is initialized.
\begin{lstlisting}
sage: G1.find_pos_domains(i,j)
\end{lstlisting}
finds all positive domains between the \verb|i|-th and the \verb|j|-th generators.
\begin{lstlisting} 
sage: G1.domain_type(i,j,D)
\end{lstlisting}
returns the $J_+$, the Euler characteristic, the number of boundary components, and the genus of a Lipshitz curve associated to the domain D between the \verb|i|-th and the \verb|j|-th generators.
\begin{lstlisting} 
sage: G1.find_pos_diffs_from(i)
\end{lstlisting}
returns a list of positive Maslov index 1 domains from the \verb|i|-th generator along with the corresponding output generator.
\begin{lstlisting} 
sage: G1.find_pos_diffs_to(j)
\end{lstlisting}
returns a list of positive Maslov index 1 domains to the \verb|j|-th generator along with the corresponding input generator.
\begin{lstlisting} 
sage: G1.plot_complex(k) 
\end{lstlisting}
returns a directed graph where the vertices are the Heegaard Floer generators in the \verb|k|-th SpinC structure and the edges are all possible differentials, with blue indicating those with unique holomorphic representatives. 


\section{A tutorial for \texttt{hf-hat-obd-nice}}
\label{sec:hf-hat-obd-nice}
Download \verb|hf-hat-obd-nice.sage| from \cite{KMVHMW-github}. After running SageMath on your computer, load the program:
\begin{lstlisting}
sage: load('hf-hat-obd-nice.sage')
\end{lstlisting}
The primary input for the program is the same as it is for \verb|hf-hat-obd.sage|; namely, a list of regions cut out by the alpha and beta curves in the (multi-)pointed Heegaard diagram resulting from an open book and a collection of disjoint properly embedded arcs on a page of the open book. We require that the Heegaard diagram is nice. To produce a nice Heegaard diagram, use \verb|makenice.sage|, which can be downloaded from \cite{KMVHMW-github}. For example, the following regions list
\begin{lstlisting}
sage: rlist = [
    [1,0,5,6],
    [2,1,6,4,6,5],
    [3,2,5,4,1,2,8,7],
    [9,8,2,3,9,7],
    [3,0,8,9],
    [0,1,4,6,4,5,0,3,7,9,7,8]]
\end{lstlisting}
does not define a nice Heegaard diagram as it has two hexagonal regions and one octagonal region. To produce a nice Heegaard diagram out of this, load \verb|makenice.sage| via
\begin{lstlisting}
sage: load('makenice.sage')
\end{lstlisting}
and run
\begin{lstlisting} 
sage: M12 = HeegaardDiagram(rlist,1)
sage: M12.compute_distances()
sage: M12.make_nice()
sage: nicerlist = M12.nice_boundary_intersections
\end{lstlisting}
After initializing the nice Heegaard diagram via
\begin{lstlisting} 
sage: N12 = NiceHeegaardDiagram(nicerlist,1)
\end{lstlisting}
and running
\begin{lstlisting} 
sage: N12.sort_generators()
\end{lstlisting}
the \verb|N12| versions of the variables (1)--(30) from Section \ref{sec:hf-hat-obd} become available. Then we are ready to use the following functions.
\begin{lstlisting}
sage: N12.find_diffs(i,j)
\end{lstlisting}
finds all differentials between the \verb|i|-th and the \verb|j|-th generators.
\begin{lstlisting}
sage: N12.print_differentials(k)
\end{lstlisting}
outputs a \verb|.txt| file (\verb|N12_differentials_in_spinc_k.txt|) with all Heegaard Floer differentials in the \verb|k|-th SpinC structure.
\begin{lstlisting} 
sage: N12.plot_complex(k) 
\end{lstlisting}
returns a directed graph where vertices are the Heegaard Floer generators in the \verb|k|-th SpinC structure and edges are the Heegaard Floer differentials between them. 
\begin{lstlisting} 
sage: N12.compute_homology(k) 
\end{lstlisting}
returns the hat version of the Heegaard Floer homology of the 3-manifold described by the Heegaard diagram in the \verb|k|-th SpinC structure. 

If you want to compute the Ozsv\'ath--Szab\'o contact invariant and the spectral order of the open book diagram, it suffices to run
\begin{lstlisting} 
sage: N12.sort_canonical_spinc()
\end{lstlisting}
instead of \verb|N12.sort_generators()| so as to sort those generators in the canonical SpinC structure that reside in Maslov gradings $0$ and $1$ relative to the canonical generator $\x_\xi$. Either will populate the lists \verb|N12.cx0| and \verb|N12.cx1| containing those generators in (relative) Maslov gradings $0$ and $1$, respectively. These are then used by 
\begin{lstlisting} 
sage: N12.compute_bmaps()
\end{lstlisting}
to compute $\dhat$, $\del_0$, and $\del_1$ between (relative) Maslov gradings $1$ and $0$,
\begin{lstlisting} 
sage: N12.check_contact_class()
\end{lstlisting}
to check whether the Ozsv\'ath--Szab\'o contact invariant is zero, and
\begin{lstlisting} 
sage: N12.compute_order()
\end{lstlisting}
to compute the spectral order $\ordn(S,\phi,\arcs)$ of the open book diagram $(S,\phi,\arcs)$.
\newpage
\bibliography{biblio}

\end{document}